\documentclass[12pt]{article}
\usepackage[english,activeacute]{babel}
\usepackage{amsmath,amsfonts,amssymb,amstext,amsthm,amscd,mathrsfs,amsbsy}

\newtheorem{teo}{Theorem}[section]
\newtheorem{pro}[teo]{Proposition}
\newtheorem{coro}[teo]{Corollary}
\newtheorem{lem}[teo]{Lemma}

\theoremstyle{definition}
\newtheorem{defi}[teo]{Definition}
\newtheorem{exam}[teo]{Example}
\newtheorem{rem}[teo]{Remark}
\newtheorem{algo}[teo]{Algorithm}

\newcommand{\V}{\mathbf V}

\newcommand{\R}{\mathbb R}
\newcommand{\N}{\mathbb N}
\newcommand{\Z}{\mathbb Z}

\newcommand{\C}{\mathbb C}

\newcommand{\Po}{\C[x_1,\ldots,x_n]}
\newcommand{\A}{\mathscr A}
\newcommand{\YA}{Y_{\A}}
\newcommand{\IA}{I_{\A}}
\newcommand{\eA}{\eta(\A)}
\newcommand{\dA}{\delta(\A)}
\newcommand{\mA}{\min(\A)}

\newcommand{\Si}{\mbox{Sing}}
\newcommand{\Co}{\mbox{Conv}}
\newcommand{\J}{\mathcal J}

\begin{document}

\title{Nash modification on toric curves}

\author{Daniel Duarte\footnote{Research supported by CONACyT.}, Daniel Green Tripp\footnote{Research supported by 
FORDECyT-265667.}}

\maketitle

\begin{abstract}
We revisit the problem of resolution of singularities of toric curves by iterating Nash modification. We give a bound on the
number of iterations required to obtain the resolution. We also introduce a different approach on counting iterations by 
dividing the combinatorial algorithm of Nash modification of toric curves into several division algorithms.
\end{abstract}


\section{Introduction}

The Nash modification of an equidimensional algebraic variety is a modification that replaces singular points by 
limits of tangent spaces. It has been proposed to iterate this construction to obtain resolution of singularities (\cite{No,S}).
So far,  only in very few cases it is known that this method do work: the case of curves (\cite{No}) and 
the family $\{z^p+x^qy^r=0\}\subset\C^3$ for positive integers $p,q,r$ (\cite{R}). In \cite{No} it is also proved 
that Nash modification do not resolve singularities in positive characteristic. 

There is a variant of this problem which consists in iterating the Nash modification followed by normalization. This variant was 
extensively studied in $\cite{GS-1,GS-2,Hi,Sp}$, culminating with the theorem by M. Spivakovsky that normalized Nash modification 
solves singularities of complex surfaces.

Several years later, the original question received a new wave of attention with the appearance of several papers exploring 
the case of toric varieties (\cite{At,GT,GM,D}). It was proved in \cite{GT,GM} that the iteration of Nash modification
for toric varieties corresponds to a purely combinatorial algorithm on the semigroups defining the toric variety.
Using this combinatorial description, some new partial results were obtained in the mentioned papers. Nevertheless, as far 
as we know, the question for toric varieties has not been completely solved so far.

In this note, we take a step backward hoping that a better understanding of the simplest case (that of toric curves) might
throw some light on the problem for higher-dimensional toric varieties. 

As we mentioned before, it is already known that Nash modification solves singularities of curves. In this note we revisit this 
problem by giving a combinatorial proof of that fact for toric curves. In addition, we give a bound on the number of iterations 
required to obtain the desingularization. We also introduce a different approach on counting iterations. As we will see, the combinatorial
algorithm for toric curves can be divided into several division algorithms that record the significant improvements that take place during
the algorithm. We will give an effective bound for the number of such division algorithms. 

We conclude with a brief discussion on some special features of Nash modification of toric curves. These features include comments 
on Hilbert-Samuel multiplicity, embedding dimension, the support of an ideal defining the Nash modification, and a possible generalization 
of our results for toric surfaces. 


\section{Nash modification of a toric variety}

Let us start by recalling the definition of Nash modification of an equidimensional algebraic variety.

\begin{defi}
Let $X\subset\C^n$ be an equidimensional algebraic variety of dimension $d$. Consider the Gauss map:
\begin{align}
G:X\setminus\Si(X)&\rightarrow G(d,n)\notag\\
x&\mapsto T_xX,\notag
\end{align}
where $G(d,n)$ is the Grassmanian of $d$-dimensional vector spaces in $\C^n$, and $T_xX$ is the
tangent space to $X$ at $x$. Denote by $X^*$ the Zariski closure of the graph of $G$. Call $\nu$ 
the restriction to $X^*$ of the projection of $X\times G(d,n)$ to $X$. The pair $(X^*,\nu)$ is called the 
\textit{Nash modification} of $X$. 
\end{defi}

We are interested in studying this construction in the case of toric curves. Let us recall the definition of an affine toric variety 
(see, for instance, \cite[Section 1.1]{CLS} or \cite[Chapter 4]{St}).

Let $\A=\{a_1,\ldots,a_n\}\subset\Z^d$ be a finite set satisfying $\Z\A=\{\sum_i \lambda_i a_i|\lambda_i\in\Z\}=\Z^d$. 
The set $\A$ induces a homomorphism of semigroups
\begin{align}\label{e. pi}
\pi_{\A}:\N^n\rightarrow\Z^d,\mbox{ }\mbox{ }\mbox{ }\alpha=(\alpha_1,\ldots,\alpha_n)
\mapsto \alpha_1 a_1+\cdots+\alpha_n a_n.\notag
\end{align}
Consider the ideal
$$I_{\A}:=\langle x^{\alpha}-x^{\beta}|\alpha,\beta\in\N^n,\mbox{ }\pi_{\A}(\alpha)=\pi_{\A}
(\beta)\rangle\subset\Po.$$

\begin{defi}
We call $\YA:=\V(\IA)\subset\C^n$ the toric variety defined by $\A$.
\end{defi}


It is well known that a variety obtained in this way is irreducible, contains a dense open set isomorphic 
to $(\C^*)^d$ and such that the natural action of $(\C^*)^d$ on itself extends to an action on $\YA$. 

\subsection{Combinatorial algorithm for affine toric varieties}\label{s. algorithm}

It was recently proved that the iteration of the Nash modification of a toric variety $\YA$ corresponds to a combinatorial algorithm 
involving the elements in $\A$ (see \cite{GT, GM}). Here we follow the results as presented in \cite[Section 4]{GM}. 

Let $\A=\{a_1,\ldots,a_n\}\subset\Z^d$ be a finite set such that $\Z\A=\Z^d$ and $0\notin\Co(\A)$, where $\Co(\A)$ is the
convex hull of $\A$ in $\R^d$. Let $\YA\subset\C^n$ be the corresponding affine toric variety. The Nash modification $\YA^*$ of 
$\YA$ can be described as follows.

Given $n\in\N$, denote $[n]:=\{1,\ldots,n\}$. Let $J=(j_1,\ldots,j_d)\in [n]^d$, where $1\leq j_1<j_2<\ldots<j_d\leq n$. 
Denote $\det(a_J):=\det(a_{j_1}\cdots a_{j_d})$. Now fix $J_0\in[n]^d$ such that $\det(a_{J_0})\neq0$ and for any other 
$J\in[n]^d$ denote $\Delta_J^{J_0}:=\sum_{j\in J}a_j-\sum_{k\in J_0}a_k.$

\begin{teo} [\cite{GM}, Section 4] \label{t. covering}
$\YA^*$ is covered by affine toric varieties $Y_{\A_{J_0}}$, where 
$\A_{J_0}=\{a_j\}_{j\in J_0}\cup\{\Delta_J^{J_0}|\#(J\setminus J_0)=1,\det(a_J)\neq0\}$ and $0\notin\Co(\A_{J_0})$.
\end{teo}

In addition, we consider the following criterion of non-singularity for toric varieties.

\begin{pro} [\cite{GM}, Criterion 3.16]\label{p. non singular}
Let $\A\subset\Z^d$ be as before. The affine toric variety $\YA$ is non-singular if and only if there are $d$ elements in $\N\A$
such that $\N\A$ is generated by those elements as a semigroup.
\end{pro}

It follows from the theorem and proposition that iteration of Nash modification of toric varieties coincides with the following 
combinatorial algorithm:

\begin{algo}\label{algo gral}
\begin{enumerate}
\item Let $\A=\{a_1,\ldots,a_n\}\subset\Z^d$ be a finite set such that $\Z\A=\Z^d$ and $0\notin\Co(\A)$.
\item For every $J_0\in[n]^d$ such that $\det(a_{J_0})\neq0$ consider those sets $\A_{J_0}$ satisfying $0\notin\Co(\A_{J_0})$.
\item If every $\N\A_{J_0}$ from step 2 is generated by $d$ elements then stop. Otherwise replace $\A$ by each $\A_{J_0}$ 
that is not generated by $d$ elements and return to step 1. 
\end{enumerate}
\end{algo}

\subsection{Combinatorial algorithm for affine toric curves}

We are interested in applying the previous algorithm for toric curves. Let us describe how the algorithm looks like in this special case.

First, we can assume that $\A=\{a_1,\ldots,a_n\}\subset\Z$ is such that $0<a_1<...<a_n$. This follows from the fact that
$0\notin\Co(\A)$ if and only if $\A\subset\N$ or $\A\subset\Z\setminus\N$ and that $Y_{\A}\cong Y_{-\A}$. In addition, the
condition $\Z\A=\Z$ is equivalent to $\gcd(\A)=1$. For these reasons, from now on, $\A\subset\Z$ will denote a finite set 
$\{a_1,\ldots,a_n\}$ such that 
\begin{equation}\label{cond}
0<a_1<\ldots< a_n \mbox{ and } \gcd(\A)=1. \tag{$\ast$}
\end{equation}

Finally, proposition \ref{p. non singular} is equivalent to ask that $1\in\A$ for toric curves.

Now we can describe algorithm \ref{algo gral} in this case. Observe that in this case only for $J_0=(1)\in[n]$ it happens 
that $0\notin\Co(\A_{J_0})$ (see step 2 of algorithm \ref{algo gral}). In other words, the Nash modification of $\YA$ is contained 
in a single affine chart according to Theorem \ref{t. covering}. 

\begin{algo}\label{algo}
\begin{enumerate} 
\item Let $\A=\{a_1,\ldots,a_n\}\subset\Z$ be a finite set satisfying (\ref{cond}).
\item Let $\A':=\{a_1\}\cup\{a_2-a_1,\ldots,a_n-a_1\}$.
\item If $1\in\A'$ the algorithm stops. Otherwise replace $\A$ with $\A'$ and return to step 1.
\end{enumerate}
\end{algo}

\begin{exam}
Let $\A=\{12,28,33\}$. Then:
\begin{align}
\A'&=\{12\}\cup\{28-12,33-12\}=\{12,16,21\},\notag\\
\A''&=\{12\}\cup\{16-12,21-12\}=\{4,9,12\},\notag\\
\A'''&=\{4\}\cup\{9-4,12-4\}=\{4,5,8\},\notag\\
\A''''&=\{4\}\cup\{5-4,8-4\}=\{1,4\}.\notag
\end{align}
Since $1\in\A''''$, the algorithm stops, that is, $Y_{\A''''}$ is a non-singular curve. 
\end{exam}

\begin{rem}
Notice the resemblance of this algorithm with Euclid algorithm (this is why in \cite{GM} algorithm \ref{algo gral} is called
\textit{multidimensional Euclidean algorithm}). Actually, for $n=2$, algorithm \ref{algo} is exactly Euclid algorithm for $a_1$ and $a_2$.
\end{rem} 

\begin{rem}\label{r. redundant}
Let $\A$ be as before. If $a_i=qa_1$ for some $q\in\N$ and $i>1$ then $\N\A=\N(\A\setminus\{a_i\})$. In particular, 
$Y_\A \cong Y_{\A\setminus\{a_i\}}$. Therefore, some times we will assume that no multiples of the minimum of $\A$ other than
itself appear in $\A$. 
\end{rem}


\section{Resolution of toric curves and number of iterations}

It is well known that the iteration of Nash modification resolves singularities of curves (\cite[Corollary 1]{No}).
In this section we revisit this problem for toric curves giving a combinatorial proof in this special case. We also consider 
the problem of finding a bound for the number of iterations required to obtain a non-singular curve.

\subsection{Resolution of toric curves}

Given a finite set $\A\subset\Z$ satisfying (\ref{cond}), we denote $\A^{1}:=\A'$, where $\A'$ is the set obtained after
applying once algorithm \ref{algo} to $\A$. Similarly, $\A^{k}:=(\A^{k-1})'$, for $k\geq2$.

\begin{lem}\label{l. one iteration}
Let $\A=\{a_1,\ldots,a_n\}\subset\Z$ be a finite set satisfying (\ref{cond}) and such that $1<a_1$. Then $\min(\A^1)\leq\min(\A)$. 
In addition, there is $k\in\N$ such that  $\min(\A^k)<\min(\A)$.
\end{lem}
\begin{proof}
From remark \ref{r. redundant}, we can assume that $\A$ contains no multiples of $\mA$. From algorithm 
\ref{algo} we know that $\A^1=\{a_1,a_2-a_1,\ldots,a_n-a_1\}$. From the assumption on $\A$, $a_1\neq a_2-a_1$. 
If $a_1<a_2-a_1$ then $\min(\A^1)=a_1=\min(\A)$. If $a_2-a_1<a_1$ then $\min(\A^1)<\min(\A)$.
This proves that $\min(\A^1)\leq\min(\A)$. 

Now we perform division algorithm for $a_1$ and $a_2$ (since $1<a_1$ then $|\A|\geq2$ according to (\ref{cond})): 
there exist $q,r\in\N$ such that $a_2=a_1q+r$, where $0<r<a_1$. Notice that for any $0<l<q$, 
$\A^l=\{a_1,a_2-la_1,\ldots,a_n-la_1\}$, where $a_1<a_2-la_1<\ldots<a_n-la_1$ and 
$\A^q=\{a_1,a_2-qa_1,\ldots,a_n-qa_1\}$, where $a_2-qa_1=r<a_1$ and $a_2-qa_1<a_3-qa_1<\ldots<a_n-qa_1$.
Hence, $\min(\A^q)=a_2-qa_1=r<a_1=\min(\A)$.
\end{proof}

\begin{rem}
The proof of the previous lemma shows that the smallest $k$ such that $\min(\A^k)<\min(\A)$ can be obtained from
the division algorithm applied to the two smallest elements of $\A$. 
In particular, algorithm \ref{algo} can be divided into several division algorithms that record the significant improvements
of the algorithm (see section \ref{s. counting}).
\end{rem}

\begin{rem}
It is also important to observe (following the notation of the proof of the lemma) that $\min(\A^q\setminus\{r\})$ is not necessarily 
$a_1$ (consider for example $\A=\{7,17,19\}$). 
\end{rem}

\begin{pro}\label{p. Nash resolves}
A finite iteration of Nash modification resolves singularities of toric curves.
\end{pro}
\begin{proof}
Let $\YA$ be a toric curve defined by $\A=\{a_1,\ldots,a_n\}\subset\Z$ satisfying (\ref{cond}) and such that $1<a_1$. 
Assume that $\A$ contains no multiples of $\mA$.
By the division algorithm, $a_2=a_1q+r$, where $0<r<a_1$. From the proof of lemma \ref{l. one iteration}, it follows that 
$1\leq\min(\A^q)<a_1$. If $1=\min(\A_q)$ then $ Y_{\A^q}$ is non-singular. Otherwise repeat the process for $\A^q$. Continuing
like this we obtain a sequence $a_1>\min(\A^q)>\min(\A^{q_1})>...\geq1$. This decreasing sequence cannot be infinite, so $1\in\A^k$,
for some $k$.
\end{proof}

\begin{rem}
There is a geometric interpretation of the number $\min(\A)$ in terms of Hilbert-Samuel multiplicity (see section \ref{s. features}).
\end{rem}

\subsection{Number of iterations}

Now we consider the problem of giving a bound to the number of iterations required to resolve a toric curve using Nash modification.

\begin{defi}
Let $\A\subset\Z$ be a finite set satisfying (\ref{cond}).
We denote as $\eta(\A)$ the number of iterations required to solve the singularities of $\YA$ using Nash modification. 
In other words, $\eta(\A)$ is the minimum $k\in\N$ such that $1\in\A^k$.
\end{defi}

In general, to estimate $\eA$ is not a simple task. Let $\A:=\{a_1^{(0)},\ldots,a_{n_0}^{(0)}\}$ be such that 
no multiple of $\mA$ (other than itself) is contained in $\A$.
Following the proof of lemma \ref{l. one iteration}, algorithm \ref{algo} can be summarized as follows 
(at every row we assume that $\A^q_i$ contains no multiples of $\min(\A^q_i)$ and its elements are 
ordered increasingly):
\begin{align}\label{algo res}
\A&=\{a_1^{(0)},a_2^{(0)},\ldots,a_{n_0}^{(0)}\}, 	
&&a_2^{(0)}=a_1^{(0)}q_1+r_1,\notag\\
\A^{q_1}&=\{a_1^{(1)},a_2^{(1)},\ldots,a_{n_1}^{(1)}\}, 
&&a_1^{(1)}=r_1,\mbox{ }\mbox{ } a_2^{(1)}=a_1^{(1)}q_2+r_2,\notag\\
\A^{q_2}&=\{a_1^{(2)},a_2^{(2)},\ldots,a_{n_2}^{(2)}\},	
&&a_1^{(2)}=r_2, \mbox{ }\mbox{ }a_2^{(2)}=a_1^{(2)}q_3+r_3,\\
&\mbox{ }\mbox{ }\vdots 				&&\mbox{ }\mbox{ }\mbox{ }\mbox{ }\vdots\notag\\
\A^{q_{s-1}}&=\{a_1^{(s-1)},a_2^{(s-1)},\ldots,a_{n_{s-1}}^{(s-1)}\},	
&&a_1^{(s-1)}=r_{s-1},\mbox{ }\mbox{ } a_2^{(s-1)}=a_1^{(s-1)}q_{s}+1.\notag\\
\A^{q_s}&=\{1\}.\notag
\end{align}

Summarized like this, $\eA=q_1+q_2+\cdots+q_s$, by definition. The input of the algorithm is $\A$, thus we can compute 
$q_1$ directly from it. Unfortunately, for $i\geq2$, it is not clear how to bound the numbers $q_i$ since, to begin with, there 
is no control over the values of $a_1^{(i)}$ and $a_2^{(i)}$. Those values depend on how $a_2^{(0)},\ldots,a_{n_0}^{(0)}$ 
are distributed.



There is, however, another way to bound $\eA$. The following proposition was adapted from \cite[Lemma 5.5]{D}.

\begin{pro}\label{p. Nash resolution}
Let $\A=\{a_1,\ldots,a_n\}\subset\Z$ be such that (\ref{cond}) holds. Let
$$v(\A):=\min\{a_i|\gcd(a_1,\ldots,a_i)=1\}.$$
Then $v(\A^1)\leq v(\A)-2$. Therefore, if $a_1>1$, after at most $\lfloor\frac{v(\A)}{2}\rfloor$ iterations of the algorithm, the resulting 
set contains 1. In other words, $\eA\leq\lfloor\frac{v(\A)}{2}\rfloor$.
\end{pro}
\begin{proof}
Assume that $1<a_1$. Let $a_{k}=v(\A)$, where $2\leq k\leq n$. By applying once the algorithm we obtain, in particular, 
$\{a_1,a_2-a_1,\ldots,a_k-a_1\}\subset\A^1$. Call $N=\max\{a_1,a_k-a_1\}$. Since $\gcd(a_1,a_2-a_1,\ldots,a_k-a_1)=1$ we have 
$v(\A^1)\leq N$. If $N=a_k-a_1$ then, since $a_1\geq2$ we have $v(\A^1)\leq a_k-a_1\leq v(\A)-2$. Suppose now that $N=a_1$. 
If $a_k=a_1+1$ then $a_k\geq3$ and since $1=a_k-a_1\in\A^1$ we obtain $v(\A^1)=1<3\leq a_k=v(\A)$, so $v(\A^1)\leq v(\A)-2$.
Otherwise $a_k>a_1+1$ which implies $v(\A^1)\leq a_1\leq a_k-2=v(\A)-2$. This proves the proposition.
\end{proof}

Unfortunately, the bound given in the previous proposition for $\eA$ is not accurate in general, as the following example shows.

\begin{exam}
Let $\A=\{10^{10},2\cdot 10^{10}+1\}$. Then $\A^1=\{10^{10},10^{10}+1\}$ and $\A^2=\{1,10^{10}\}$. Thus, it took two 
iterations to terminate the algorithm but the bound is $\lfloor\frac{v(\A)}{2}\rfloor=10^{10}$.
\end{exam}


Eventhough we did not succeed in giving a more precise bound for $\eA$, the summary of algorithm \ref{algo} given in (\ref{algo res})
shows that the algorithm can be divided into several division algorithms. In the next section we give a bound for the number of those
division algorithms (the number $s$ in the notation of (\ref{algo res})). As we will see, this is a simpler task.


\section{Counting division algorithms}\label{s. counting}

Let us start with an example.

\begin{exam}
In this example, it takes twelve iterations of algorithm \ref{algo} to get the element 1 which can be summarized in three 
division algorithms.
\begin{align}
\A&=\{20,165,172\},\notag\\		
\A^8&=\{20,165- 20\cdot8, 172-20\cdot8\}=\{5,12,20\},\notag\\
\A^{10}&=\{5,12-5\cdot2,20-5\cdot2\}=\{2,5,10\}\notag\\
\A^{12}&=\{2,5-2\cdot2,10-2\cdot2\}=\{1,2,6\}.\notag
\end{align}
\end{exam}

\begin{defi}
Let $\A=\{a_1,\ldots,a_n\}\subset\Z$ be such that (\ref{cond}) holds. Assume that $1<a_1$. We denote as $\delta(\A)$ the 
number of division algorithms required for algorithm \ref{algo} to stop (this corresponds to number $s$ in (\ref{algo res})). 
\end{defi}

\begin{rem}
Notice that $\dA\leq\eA$.
\end{rem}

The goal of this section is to give a bound for $\dA$. As we will see, the bound depends only on the first two elements of $\A$. 
First we need to introduce the Fibonacci sequence.

\begin{defi} The \textit{Fibonacci sequence} is defined as
$$F_1 := 1, \ F_2 := 2 \mbox{ and }  F_{m+2} := F_{m+1} + F_{m} \mbox{ for each }  m\geq 1.$$
\end{defi}

The following theorem is a direct generalization of the analogous result for $n=2$, which seems to be well known
(see, for instance, \cite{Gr}).

\begin{teo}\label{t. bound}
Let $\A=\{a_1,\ldots,a_n\}\subset\Z$ be such that (\ref{cond}) holds and $1<a_1$. Then $a_1\geq F_{\delta(\A)+1}$ and 
$a_2 \geq F_{\delta(\A)+2}$.
\end{teo}
\begin{proof}
We proceed by induction on $\delta(\A)$. For $\delta(\A)=1$ the result is true since, by hypothesis, $a_1\geq2=F_{1+1}$ and 
$a_2\geq a_1+1\geq2+1=3= F_{1+2}$. 

Assume that the statement is true for all $\A$ such that $\delta(\A)=m$. Suppose that $\delta(\A)=m+1$. Assume that
$\A$ does not contain multiples of $\min(\A)$. Applying the division algorithm to $a_1$ and $a_2$ we obtain $a_2=a_1q+r$, where
$0<r<a_1$. Then $\delta(\A^q)=m$ and $\min(\A^q)=r$. By induction, 
$r\geq F_{m+1}$. Let us prove that $a_1\geq F_{m+2}$. Since $a_1>r$ then $a_1\geq\min(\A^q\setminus\{r\})$. On the other 
hand, by induction we know that $\min (\A^q \setminus\{r\})\geq F_{m+2}$. Consecuently, $a_1 \geq F_{m+2}$.

Thus, $r\geq F_{m+1}$ and $a_1\geq F_{m+2}$. Finally, since $q \geq 1$ it follows 
$a_2=a_1q+r\geq a_1+r\geq F_{m+2}+F_{m+1}=F_{m+3}$. Therefore, $a_1\geq F_{m+2}=F_{\delta(\A)+1}$ and 
$a_2 \geq  F_{m+3} = F_{\delta(\A)+2}$. 
\end{proof}

\begin{coro}
Let $\A=\{a_1,\ldots,a_n\}\subset\Z$ such that (\ref{cond}) holds. Assume $1<a_1$. Let $m\in\N$ be such that $a_1<F_{m+1}$ or $a_2<F_{m+2}$. 
Then $\delta (\A) < m$.
\end{coro}
\begin{proof}
Suppose that $\delta(\A)\geq m$. By the previous theorem $a_1\geq F_{\delta(\A)+1}$ and $a_2\geq F_{\delta(\A)+2}$. 
On the other hand, $F_{\delta(\A)+1}\geq F_{m+1}$ and $F_{\delta(\A)+2}\geq F_{m+2}$ implying $a_1\geq F_{m+1}$ 
and $a_2 \geq F_{m+2}$. 
\end{proof}

Eventhough the previous corollary gives a good bound for $\dA$, in practice it may not be easy to find the smallest $m$ such that 
$a_1<F_{m+1}$ or $a_2<F_{m+2}$ (especially for large numbers). Now we give a simpler bound using the number of digits in
$a_1$ and $a_2$ (written in base 10). We use the following well-known fact  (see, for instance, \cite{Gr}).

\begin{lem}\label{l. digits}
Let $k \geq 1$. If $m\geq5k$ then $F_{m+1}$ has at least $k+1$ digits.
\end{lem}

\begin{coro}
If $a_1$ and $a_2$ have $k_1$ and $k_2$ digits, respectively, then $\dA<\min\{5k_1,5k_2-1\}$.
\end{coro}
\begin{proof}
Let $m:= \dA$ and assume that $m\geq 5k_1$. By theorem \ref{t. bound}, $a_1\geq F_{m+1}$ and by lemma \ref{l. digits}, $F_{m+1}$ 
has at least $k_1+1$ digits. Consequently, $a_1$ has at least $k_1 + 1$ digits. This shows that if $a_1$ has $k_1$ digits then $\dA<5k_1$.
Similarly, if we assume that $m\geq 5k_2-1$ then $m+1\geq 5k_2$. Theorem \ref{t. bound} and lemma \ref{l. digits} imply that $a_2$ has 
at least $k_2 + 1$ digits. Therefore, if $a_2$ has $k_2$ digits then $\dA<5k_2-1$.
\end{proof}



\section{Some other features}\label{s. features}

In this last section we collect some other features regarding the Nash modification of toric curves. We will see that there are some 
nice properties of the Nash modification of toric varieties that hold only in dimension 1.

\subsection{Hilbert-Samuel multiplicity}

It is well known that Hilbert-Samuel multiplicity plays a fundamental role in theorems of resolution of singularities. This invariant 
has been used to measure improvements on singularities after a suitable blowup. What about for the Nash modification?

The proof we gave for the resolution of singularities of toric curves iterating Nash modification was a purely combinatorial one and
the improvements in the algorithm were measured by looking at the semigroup itself. Because of the following result (which seems 
to be well known, see for instance \cite{Sh}) the improvements on the singularities of toric curves after Nash modification can also
be measured using Hilbert-Samuel multiplicity.

\begin{pro}
Let $\A=\{a_1,\ldots,a_n\}\subset\Z$ be a finite set satisfying (\ref{cond}). Let $R$ be the localization of $\C[\YA]$ at the maximal 
ideal corresponding to $0\in\YA$. Then $mult(\YA,0)=a_1$, where $mult(\YA,0)$ is the Hilbert-Samuel multiplicity of $R$.
\end{pro}

Now lemma \ref{l. one iteration} can be restated as follows.

\begin{coro}
Let $\YA$ be a singular toric curve defined by a set $\A\subset\Z$ satisfying (\ref{cond}). Then $mult(Y_{\A^1},0)\leq mult(\YA,0)$. 
In addition, the Hilbert-Samuel multiplicity drops after a finite number of iterations of the Nash modification. 
\end{coro}

Unfortunately, for higher-dimensional toric varieties, the previous corollary is false in general.

\begin{exam}
Let $a_1=(1,0)$, $a_2=(1,1)$, and $a_3=(3,4)$. The set $\A=\{a_1,a_2,a_3\}$ defines the normal toric surface 
$\YA=\V(xz-y^4)\subset\C^3$. Since $\A\subset\Z^2$ satisfies $\Z\A=\Z^2$ and $(0,0)\notin\Co(\A)$, we can apply 
algorithm \ref{algo gral} to compute the Nash modification of $\YA$.

Applying step 2 for $J_0=(1,2)\in\{1,2,3\}^2$, we obtain $\A_{J_0}=\{a_1,a_2\}\cup\{(2,3),(2,4)\}$. 
The toric ideal asociated to $\A_{J_0}$ is $I_{\A_{J_0}}=\langle y^2w-z^2,xw-yz,xz-y^3 \rangle.$ A direct computation
shows that $(0,0,0,0)\in Y_{\A_{J_0}}\subset\C^4$ is a singular point. Now notice that $mult(\YA,(0,0,0))=2$. 
Using $\mathtt{SINGULAR}$  $\mathtt{4}$-$\mathtt{0}$-$\mathtt{2}$ (\cite{DGPS}), we computed the Hilbert-Samuel 
multiplicity of $Y_{\A_{J_0}}$ to find $mult(Y_{\A_{J_0}},(0,0,0,0))=3$. Thus, there is a point in the preimage of Nash 
modification of $(0,0,0)\in\YA$ whose Hilbert-Samuel multiplicity increases.
\end{exam}

\subsection{Embedding dimension}

It is well known that Nash modification of an algebraic variety may not preserve embedding dimension. A. Nobile
illustrated this fact with a plane curve $X\subset\C^2$ given by the parametrization $x=t^4$, $y=t^{11}+t^{13}$. 
Nobile proved that at the only point of $\nu^{-1}((0,0))$, its embedding dimension is 3 (see \cite[Example 3]{No}). 
In the following proposition we see that for toric curves the embedding dimension do not increase after applying Nash 
modification.

\begin{pro}
Let $\A=\{a_1,\ldots,a_n\}\subset\Z$ be a finite set satisfying (\ref{cond}). Let $\A'\subset\Z$ be the set obtained from
algorithm \ref{algo}. Then the embedding dimension of $Y_{\A'}$ is less or equal than the embedding dimension of $\YA$.
\end{pro}
\begin{proof}
Suppose that the semigroup $\N\A$ is minimally generated by $\A$. Then $\YA$ has embedding dimension $n$. According 
to algorithm \ref{algo}, $|\A^1|\leq n$. In particular, the embedding dimension of $\YA^*=Y_{\A^1}$ is less or equal than $n$.
\end{proof}

The previous proposition is not true for general toric varieties as the following example shows.

\begin{exam}
Let $\A=\{(1,0),(1,1),(2,3)\}\subset\Z^2$. Let $\YA=\V(xz-y^3)\subset\C^3$ be the corresponding toric surface. The 
embedding dimension of $\YA$ is 3. Applying the combinatorial algorithm \ref{algo gral} to $\A$ we obtain the following two sets in $\Z^2$, 
which determine the affine charts covering $\YA^{*}$: $\A_1=\{(1,0),(1,1),(1,2),(1,3)\}$ and $\A_2=\{(-1,-2),(0,-1),(1,1),(2,3)\}$.
A direct computation shows that the semigroups $\N\A_1$ and $\N\A_2$ are minimally generated by $\A_1$ and $\A_2$, respectively.
Thus, the embedding dimension of $Y_{\A_1}\subset\C^4$ or $Y_{\A_2}\subset\C^4$ at the origin is 4. We conclude that there are
points in $\YA^{*}$ on which the embedding dimension increases.
\end{exam}

\subsection{Zero locus of an ideal defining the Nash modification}

Let $\A=\{a_1,\ldots,a_n\}\subset\Z^d$ be a set of vectors defining a toric variety $\YA\subset\C^n$.
It was proved in \cite{GS-1,LJ-R,GT} that an ideal whose blowup defines the Nash modification of a toric variety
can be described in the following combinatorial way (this ideal is usually called \textit{logarithmic jacobian ideal}).
\begin{teo}\label{t. ideal toric Nash}
Let $\J$ be the ideal of the coordinate ring $\C[u_1,\ldots,u_n]/\IA$ of $\YA$ generated by the products 
$u_{i_1}\cdots u_{i_d}$ such that $\det(a_{i_1}\cdots a_{i_d})\neq0$. Then the Nash modification of $\YA$ is 
isomorphic to the blowup of $\YA$ along the ideal $\J$.
\end{teo}
Notice that in the case of curves, where $\A\subset\Z$ satisfies (\ref{cond}), 
$\J=\langle u_1,\ldots,u_n \rangle\subset\C[u_1,\ldots,u_n]/\IA$. 
\begin{coro}
The Nash modification of a singular toric curve $\YA$ is isomorphic to blowing up an ideal whose zero locus is the origin in $\YA$.
\end{coro}

In other words, the zero locus of the ideal $\J$, whose blowup defines the Nash modification of the toric curve $\YA$,
satisfies $\Si(\YA)=\V(\J)$. For toric surfaces this is not necessarily true. 

\begin{exam}
Let $\A=\{(1,0),(1,1),(1,2)\}\subset\Z^2$. Let $\YA=\V(xz-y^2)\subset\C^3$ be the corresponding toric surface.
In this case $\J=\langle xy,xz,yz \rangle$. Then $\Si(\YA)=\{(0,0,0)\}\subsetneq \{(x,0,z)\in\YA|x=0 \mbox{ or }z=0\}=\V(\J).$
\end{exam}

Actually, it can be proved that, under some hypothesis, for toric surfaces $\YA$ with isolated singular point, there is no choice 
of an ideal $\J$ defining the Nash modification of $\YA$ such that $\Si(\YA)=\V(\J)$ (see \cite[Theorem 3.6]{ChD}).

\subsection{Possible generalizations for toric surfaces}

The results presented in this note were motivated by the same questions for toric surfaces: does Nash modification
resolves singularities of toric surfaces? and, if so, how many iterations does it take? 

In \cite{D} both questions were explored and some partial answers were given (in \cite{GT} the first question was also partially 
answered for toric varieties of any dimension). Unfortunately, the bounds given in \cite{D} are not at all accurate. On the other hand, 
the cases considered in that paper were studied by separating algorithm \ref{algo gral} into two independent algorithms: one algorithm 
in $\{0\}\times\Z$ followed from an algorithm in $\Z\times\{0\}$. Both of these two algorithms are quite similar to the one we studied in 
this note. 

Because of this resemblance, we strongly believe that the approach of counting division algorithms that we developed 
in section \ref{s. counting} can be applied in dimension 2, at least in the cases considered in \cite{D}. This way of counting 
iterations might result in a better understanding of those cases for which we still do not know whether Nash modification solves
singularities.

\end{document}